\documentclass[letterpaper,13pt]{amsart}

\usepackage{amssymb,amsbsy,amsmath,amsfonts,amssymb,amscd}
\usepackage{latexsym}
\usepackage{graphics}
\usepackage{color}
\usepackage{comment}
\input xy
\xyoption{all}

\theoremstyle{plain}
\newtheorem{thm}{Theorem}[section]
\newtheorem{theorem}[thm]{Theorem}

\newtheorem{corollary}[thm]{Corollary}
\newtheorem{proposition}[thm]{Proposition}
\theoremstyle{definition}
\newtheorem{remark}[thm]{Remark}

\newtheorem{definition}[thm]{Definition}

\newtheorem{example}[thm]{Example}

\newtheorem{conjecture}[thm]{Conjecture}

\newtheorem{question}[thm]{Question}

\numberwithin{equation}{section}
\newcommand{\sA}{{\mathcal A}}
\newcommand{\sB}{{\mathcal B}}

\newcommand{\sF}{{\mathcal F}}
\newcommand{\sG}{{\mathcal G}}
\newcommand{\sH}{{\mathcal H}}

\newcommand{\sT}{{\mathcal T}}
\newcommand{\sU}{{\mathcal U}}
\newcommand{\sV}{{\mathcal V}}

\newcommand{\PP}{\ensuremath{\mathbb{P}}}

\newcommand{\CC}{\ensuremath{\mathbb{C}}}
\newcommand{\RR}{\ensuremath{\mathbb{R}}}
\newcommand{\ZZ}{\ensuremath{\mathbb{Z}}}
\newcommand{\QQ}{\ensuremath{\mathbb{Q}}}

\newcommand{\hol}{\ensuremath{\mathcal{O}}}
\newenvironment{dedication}
        {\begin{quotation}\begin{center}\begin{em}}
        {\par\end{em}\end{center}\end{quotation}}

\newcommand\om{\omega}
\newcommand\la{\lambda}

\newcommand\al{\alpha}
\newcommand\be{\beta}
\newcommand\Ga{\Gamma}
\newcommand\De{\Delta}
\newcommand\ga{\gamma}

\newcommand\e{\epsilon}

\newcommand{\Lam}{\Lambda}

\newcommand{\ra}{\ensuremath{\rightarrow}}

\def\eea{\end{eqnarray*}}
\def\bea{\begin{eqnarray*}}

\newcommand\dual{\mathrel{\raise3pt\hbox{$\underline{\mathrm{\thinspace d
\thinspace}}$}}}
\newcommand\qe{\ifhmode\unskip\nobreak\fi\quad $\Box$}       % box for QED

\def\BOX{\hfill\lower.5\baselineskip\hbox{$\Box$}}
\newtheorem{theo}{Theorem}[section]

\newtheorem{remarkk}[theo]{Remark}

\usepackage{hyperref}
\title [Vanishing Chern classes, Manifolds Isogenous to a Torus Product.]{ Manifolds with  trivial Chern classes II:   Manifolds Isogenous to a Torus Product,  coframed Manifolds and a question by Baldassarri.}

\author{Fabrizio Catanese}
\address {Mathematisches Institut der Universit\"at Bayreuth\\
NW II,  Universit\"atsstr. 30\\
95447 Bayreuth}
\email{fabrizio.catanese@uni-bayreuth.de}
\address{  Korea Institute for Advanced Study, Hoegiro 87, Seoul, 
133--722.}
 
 \date{\today}

\begin{document}

\begin{abstract}
Mario Baldassarri asked in  1956  to describe the projective manifolds with last k Chern classes trivial  in real  cohomology.
He claimed that the solutions are Roth's  Pseudo-Abelian Varieties, while indeed the class of solutions is larger,
it includes the class introduced here of Manifolds Isogenous to a k-Torus Product: these have also vanishing Chern numbers,
and   in dimension 2    are all the solutions with $K_X$ nef.

We show that such a simple  picture  does not hold  in higher dimension.

Other solutions to Baldassarri's question are  the manifolds isogenous to k-framed or k-coframed 
 manifolds, that we  investigate here:
the  k- framed projective manifolds with   $K_X$  nef are    the Pseudo-Abelian varieties.
We show   results for  the    k-coframed  manifolds,  pose   open questions and conjectures,
treat also  the non projective case.
\end{abstract}

\maketitle

\begin{dedication}
In   memory of   Mario Baldassarri (1920-1964).
\end{dedication}

%\addtocontents{toc}{\protect\setcounter{tocdepth}{1}}
\tableofcontents
%=========================================================================

\section{Foreword and extended summary}

 Motivated by a  general question  addressed by Mario Baldassarri in  1956, we discuss characterizations of 
the Pseudo-Abelian Varieties introduced by Roth, and we introduce a first new notion, of Manifolds Isogenous to a k-Torus Product:
the latter have the  last k Chern classes trivial  in rational cohomology and vanishing Chern numbers.

We show that in dimension 2  the latter class  is  the correct substitute for some incorrect assertions by Enriques, Dantoni, Roth and Baldassarri: these are the surfaces with $K_X$ nef and $c_2(X)=0 \in H^4(X, \ZZ)$.

We observe in the last section, using a construction by Chad Schoen,   that such a simple similar picture  does not hold  in higher dimension.

We discuss then, as a class of solutions to Baldassarri's question, manifolds isogenous to
 projective (respectively: K\"ahler) manifolds whose tangent bundle or whose  cotangent bundle has a trivial subbundle
 of positive rank.
 
  We see that the  class of `partially framed' projective manifolds (that is,  whose tangent bundle  has a trivial subbundle) consists,
 in the case where $K_X$ is nef, of  the Pseudo-Abelian varieties of Roth;
while the  class of `partially coframed' projective manifolds is not yet  fully understood in spite of the 
new results that we are able to show here: and we formulate some open questions and conjectures.

In the course of the paper  we address also the  case of more general compact complex Manifolds,
introducing the new notions of suspensions over parallelizable Manifolds,  of  twisted hyperelliptic Manifolds, and we describe
 the known results  under the  
K\"ahler assumption.

\section{Introduction.}

Our  first  motivation was to  understand   a question posed in Mario Baldassarri 's paper \cite{baldassarri},
which   aimed at determining the complex projective Manifolds whose   last  Chern classes 
are zero in rational cohomology, namely  $c_i(X) = 0 \in H^{2i} (X, \QQ )$ for  $ i \geq k+1$.
He made the wrong claim  that these manifolds  are exactly the  Pseudo-Abelian varieties introduced by Roth \cite{rothPAV}, hence his work also suggested to have a better understanding   of the possible characterizations of the Pseudo-Abelian varieties.
\smallskip

Baldassarri's  claim 
 arose from the attempt to generalize  to higher dimensions  an  incorrect assertion concerning algebraic surfaces (deriving from a minor mistake in the Castelnuovo-Enriques classification theorem of  algebraic surfaces). 
 
 After clarifying the surface case,  we explain here  the state of the art in all dimensions, prove some new results,  define new classes of compact complex Manifolds, and pose some new questions.
  \bigskip

A first  basic principle that Baldassarri and his precedessors missed to consider (as 
the use of topology was not sufficiently established at the time) was the
following:
 
 \begin{remark}\label{isogeny}
 {\bf (Isogeny principle):} If we have a finite unramified map between two compact complex Manifolds
 $ f : Z \ra X$, then $c_i(Z) = 0 \in H^{2i}(Z, \QQ)$ if and only 
 if $c_i(X) = 0 \in H^{2i}(X, \QQ)$. 
 
 Defining {\bf isogeny} between Manifolds as the equivalence relation generated by the existence
 of such finite unramified maps, we see that the set of Manifolds which are solutions to Baldassarri's question consists of 
 a union of isogeny classes.
 
 \end{remark} 
 
  Baldassarri's  problem has been  solved in the extremal case where $k=0$: the characterization of the compact K\"ahler Manifolds with all Chern classes zero in rational (or real) cohomology 
  was solved  in 1978, thanks to Yau's celebrated theorem \cite{yau} on the existence of 
 K\"ahler-Einstein metrics on Manifolds with $c_1 (X) = 0 \in H^2(X, \QQ)$.
  From this, as explained in Kobayashi's book, page 116 of \cite{kobayashi},  follows that the 
compact  K\"ahler manifolds   (cKM) with $c_1(X) = c_2(X)= 0 \in H^*(X, \QQ)$
are the ones isogenous to complex tori
 (they are the so-called {\bf Hyperelliptic Manifolds}, quotients $X = T/G$
 of a complex torus $T$ by a finite group $G$ acting freely on $T$).
 Indeed, once one knows that we have a  K\"ahler-Einstein metric, that the manifold with $c_1(X) = c_2(X)= 0 \in H^*(X, \QQ)$
is flat was proven by Apte
\cite{apte} in the 50's.

 Hence, more precisely, we have:

 \begin{theorem}\label{zerochern}(Yau-Apte) 
  A compact K\"ahler manifold $X$
 such that $ c_1(X)= 0,  c_2(X) = 0, $ in $ H^*(X, \RR)$, is a  Hyperelliptic manifold. 
  \end{theorem}

\medskip
What happens for $k \geq 1$?

The first significant case is  the case where  $n : = dim(X)=2, k=1$.

Already to describe  this case we need some definition: a  compact K\"ahler Manifold
$X$ is said to be a {\bf torus product} iff $ X \cong  T \times Y$, where $T$ is a complex torus
of dimension $ \geq 1$.

Whereas one says that $X$ is {\bf strongly isogenous to a torus product }
if $ X = (T \times Y)/G$, where $G$ acts on $T, Y$
and freely by its  diagonal action on $T \times Y$ (we shall see later that, if $X$ is   a compact K\"ahler 
non uniruled Manifold, the above two notions are equivalent).

The case of Roth's {\bf Pseudo Abelian Varieties } is the special case 
where $ X = (T \times Y)/G$, $X$ is projective, and moreover $G$ acts on $T$
as a group of translations.

 In the case   $n =2, k=1$,    there is  a complete answer, which we find 
 important to state here for historical and  conceptual reasons (see Theorems \ref{crucial}, \ref{detail} and  \ref{dantoni} for more detailed statements):

 \begin{theorem}\label{1-dantoni}
 Let $X$ be a  compact K\"ahler surface  with  $c_2(X)=0 \in H^4(X, \ZZ)$. Then:
 
 \begin{enumerate}
 \item 
 if $K_X$ is nef,  then $X$ is strongly isogenous to a torus product, namely $ X = (T \times Y)/G$,
 where $G$ acts freely via a diagonal action on $ T \times Y$;
 \item
 if $X$ is a minimal surface, but $K_X$ is not nef, then $X$ is a $\PP^1$-bundle
 over an elliptic curve $E$;
 \item
 if $X$ is non minimal, then $X$ is a blow up of a $\PP^1$-bundle
 over  a curve $C$ of genus $ g \geq 2$.

 \end{enumerate} 
In the first two cases $c_1(X)^2 = c_2(X)=0$, while
in the last case $ \chi(X) = 1-g$, hence $ c_1(X)^2 = K_X^2 < 0$.
 
 If $c_1(X)=0 \in H^2(X, \RR)$, then $S$ is either a complex torus, or 
 a hyperelliptic surface.
 
 If $c_1(X)=0 \in Pic(X)$, then $X$ is a complex torus.

 \end{theorem}
 There are three features in our above theorem for $n=2$:
 
 \begin{itemize}
 \item
 (I): in case (1), where $K_X$ is nef, $X$ need not be a Pseudo-Abelian variety;
 \item
 (II): since $\chi(\hol_X) \leq 0$, then $X$ always possesses a non zero holomorphic
 1-form $ \om \in H^0(\Omega^1_X)$;
 \item
 (III): if moreover $X$ is minimal, then $\om$ is everywhere non vanishing;
 \item
 (IV): if  $X$ is minimal, then $X$ is birationally covered by a  trivial family of 
 (positive dimensional) complex tori,
 that is, there exists a dominant rational map $\psi : T \times Y \dasharrow X$.

 \end{itemize}
 
 The first feature  (I) contradicts results of Dantoni \cite{dantoni} and 
  Enriques  \cite{enriques-pg0}, (indeed Enriques' error is also reproduced in the classification theorem of Castelnuovo and Enriques
\cite{classificazione}). These papers put Baldassarri and  Roth on the wrong track.

The second feature (II) is the one which, as we shall see,  fails to hold true in dimension $n \geq 3$ as a consequence of vanishing of the top Chern class; while, if (II) holds true, since $(-1)^n c_n(X)$ is the expected number of zeros of a holomorphic 1-form, then (III) looks plausible.

If one takes (III) 
  as an assumption,  that there exists a holomorphic 1-form without zeros,
 then Baldassarri's claim holds at least in dimension 3, as shown by \cite{haoschreieder1},  which indeed uses a much weaker assumption, namely the existence of a real closed 1-form without zeros (see also 
 a similar result in Theorem \ref{coframed}).
 
 The failure of (II) for $n \geq 3$ is due to work of Chad Schoen \cite{schoen},
 which  shows  that in dimension $3$ there are simply connected manifolds, actually Calabi Yau manifolds ($K_X$ trivial in $Pic^0(X)$)
 with $c_3(X)=0$.  
 
 Concerning the failure of feature (IV) for $n \geq 3$,
 our result  here is  that the Schoen threefolds are not birationally covered by a family of  isomorphic Abelian surfaces, see Proposition \ref{schoenstrong}, and  that they do not admit a fibration onto a surface  with general fibres 
 isomorphic to a fixed 
 elliptic curve $E$, see Proposition \ref{schoen}. 
 Hence  that they are also birationally far away from being isogenous 
 to a torus product.

The first conclusion that we reach  is therefore that,  in higher dimensions, the right
assumption to make  is   feature (III).

And the guiding observation is that  an easy  class of solutions to Baldassarri's question is  provided by the manifolds $X$ which 
are isogenous  {\bf to a partially framed or to a partially co-framed}  manifold $Z$.

 Our definition here is that $Z$ is  partially (tangentially) framed if  the tangent bundle $\Theta_Z$
admits a maximal trivial subbundle $\hol_X^k$, with $k>0$.

Whereas  $Z$ is  said to be partially co-framed (cotangentially framed) if the  cotangent bundle $\Omega^1_Z$ admits such a maximal trivial subbundle $\hol_X^k$, with $k>0$.

The first definition brings back into play Roth's Pseudo-Abelian Varieties,
see Theorem \ref{framed}, which contains   results also in the non projective case:

\begin{theorem}
A tangentially k-framed projective Manifold with $K_X$ nef is a pseudo-Abelian variety,
in particular $\Theta_X \cong \hol_X^k \oplus \sF$.

\end{theorem}

The previous classical Theorem was essentially proven by Roth
and, thanks to the work of Fujiki \cite{fujiki} and Lieberman \cite{lieberman},
 the class of partially (tangentially) framed projective manifolds is understood.
 We also observe, using results of  \cite{lieberman}, \cite{amoros} (which also contains an excellent survey of the theory of framed K\"ahler manifolds), that the picture becomes more complicated if we enlarge our consideration
 to the wider realm of cKM (compact K\"ahler Manifolds), where more complicated constructions
 such as Seifert fibrations, principal torus bundles and torus suspensions enter into  the picture
 (see section 4).
 
 We also consider in Section 3 the non K\"ahler case, where we define a more general notion of suspension over a parallelizable Manifold,  for which a splitting of the type 
 $\Theta_X \cong \hol_X^k \oplus \sF$ holds true. 
 
 And we consider all the Manifolds isogenous to a parallelizable Manifold,
 calling  them {\bf twisted Hyperelliptic Manifolds} since they also enjoy the property of
 having all the real Chern classes equal to zero (but a general converse result is missing).
 \bigskip
 
 Passing to     the second class,  of partially co-framed  manifolds, we see how this interesting class  is more mysterious and presents 
some intriguing questions  (see the discussion following Proposition \ref{nozero} and Theorem \ref{coframed}). 

A new result which we prove here is the following

 \begin{theorem}\label{coframed}
 
 Assume that $X$ is a $k$-coframed compact K\"ahler manifold, such  that either
 
 (a) $q: = h^0(\Omega^1_X)=k$,
 
  or  more generally 
  
  (b) the framing is given by $f^* ( \Omega^1_A)$, where $f : X \ra A$ is holomorphic
  and $A$ is a complex torus of dimension $k$ ($f$ is the Albanese map $a_X$ in the first case).
  
  (i)  Then $f$    is a differentiable fibre bundle.
 
(ii)  If   $X$ is projective,  and $K_X$ is nef, then $X$ is a pseudo-Abelian variety provided one of the following assumptions is satisfied:

\begin{enumerate}
\item
$n-k \leq 3$,
\item
$K_X$ is pseudoample (there is a  positive integer $m$ such that  $|m K_X|$ is base point free),
\item
the fibres of $f$ have a pseudoample canonical system,
\item
$k \leq 5$.
\end{enumerate}

 \end{theorem}

 A partial   analogue of (1) holds also for  compact K\"ahler manifolds, in view of \cite{c-h-p}.
 
 \smallskip
 
For the proof of  the previous Theorem \ref{coframed} we use a recent result by Taji \cite{taji}, proving birational isotriviality
 under the more general assumption that the fibres admit a good minimal model.

 In the projective case with $K_X$ nef we do not have yet more examples of $k$-coframed Manifolds 
   than the class of 
the Pseudo-Abelian varieties.

  Observe that hypothesis (2) holds  if the so called Abundance conjecture is true, hence
 in particular it is reasonable to pose the following

  \begin{conjecture}\label{co-framed}
 
Assume that $X$ is a $k$-coframed projective  manifold with  $q =k$ and 
with $K_X$  nef: then $X$ is a pseudo-Abelian variety.
 
 \end{conjecture}
 
\medskip

In fact, we ask the more  general 

\begin{question} Are   the 
 $k$-coframed projective manifolds with $K_X$ nef  
 just the Pseudo-Abelian varieties?
 \end{question}

 It was observed by the referee that related conjectures are contained in the preprint \cite{CCHao}.

\smallskip

Finally, while    for $k$- framed Manifolds and  for $k$-coframed
Manifolds   the  last  $k$   Chern classes are   zero (in the Chow ring
if the Manifolds are projective), we observe that more generally 
 for  the Manifolds  isogenous to a k-Torus product  (in view of Remark \ref{isogeny})  not only
the last  $k$   Chern classes are   zero in rational  cohomology,  
but also   all the Chern numbers of $X$ are equal to zero. 

Hence in  the last section we briefly discuss  Manifolds with $K_X$ nef 
and with vanishing Chern numbers.

For $n=2$, as we already illustrated,   there are only  manifolds isogenous to a torus product;  but, in dimension $n=3$, in spite of a  recent result by  Hao-Schreieder \cite{haoschreiederbmy},  (valid in all dimensions, see Theorem \ref{hsbmy}) which  in particular describes      threefolds
with  $c_1 c_2=0$,  and Kodaira dimension $2$, as being birationally isogenous 
to a torus product, 
  we have also the Schoen threefolds,  which do not exhibit this feature.

Section 6 contains also some historical comments which may be useful in order to relate the present results
to  the classical literature.

\bigskip
\section{Surfaces with second Chern class equal to zero, and a Chern class characterization of Hyperelliptic and Abelian surfaces}

In  this section we establish Theorem \ref{1-dantoni}
mentioned in the Introduction: this is done rerunning the proof of the classification of surfaces.

The classification theorem by Castelnuovo and Enriques \cite{classificazione} was extended by Kodaira \cite{kodaira},
and a crucial result  is the following
(see   \cite{bea}, Chapter 6, Theorem VI.13  containing the basic   ingredients,     Theorem 1.4 of \cite{fcbl},  the crucial Theorem of \cite{time}  and
also another proof in  \cite{haoschreiederbmy}):

\begin{theorem}\label{crucial}
Let  $ S $ be a  compact smooth complex  surface, minimal in the strong sense that $K_S$ is  nef, and 
such that $\chi(S)=0$ (which implies that  $ K_S^2=0$  and the topological Euler number
$e(S) =  c_2 (S) = 0$). Equivalently,  assume that  $K_S$ is nef, and that $e(S) = c_2(S)=0$.

 Then  $X$ is strongly isogenous to a torus product, namely $ X = (T \times Y)/G$,
 where $T$ is a torus of positive dimension (1 or 2) and $G$ acts freely via a diagonal action on $ T \times Y$.

\end{theorem} 
 \medskip
 
 More precisely, 
 
 \begin{theorem}\label{detail}
 The surfaces strongly isogenous to a    torus product are classified as follows: 

1)  $p_g(S)=1$,  $q(S)=2$, and $S$ is  a complex torus $A$ (a {\bf hyperelliptic surface of grade  1}), or 

2)  $p_g(S)= p,  q(S) = p + 1$  and  $S$ is {\bf isogenous} to an elliptic  product, i.e. $S$ is the quotient  $(C_1 \times C_2)/G$ of a product of  curves
of genera $$g_1: = g (C_1) =1 , g_2: = g (C_2) \geq 1,$$
 by a free action of a finite group of product type
(that is, $G$ acts faithfully on $C_1, C_2$ and we take the diagonal action $ g (x,y) : = ( g x, gy)$),
 such that,  if we denote by $g'_j = g (C_j /G)$, then $$g'_1 + g'_2 = p + 1.$$ 

 Case 2) bifurcates into two subcases: 

(2.1,p) {\bf pseudo-elliptic case}: $g'_1=1$ (hence  $G$ acts on $C_1$ by translations),    $ C_2 / G $ has genus $p$, 
and we assume \footnote{to exclude that we are in case 1)}, for $p=1$, that $g_2 \geq 2$; or 

(2.0,p)  $g'_1=0$ (hence $C_1 /G \cong  \PP^1$),   $C_2 /G$ has genus $p+1= q(S)$,
and we assume  \footnote{to exclude that we are in case (2.1,0)}, for $p=0$, that $g_2 \geq 2$; 
here  the image of Albanese map $ \alpha : S \ra Alb(S) $ equals $C_2 /G \subset  Alb(S)$.

\medskip

Case (2.1,0) with $g_2=1$ is the case where $S$ is a properly  hyperelliptic  
(bielliptic)  surface (a {\bf  hyperelliptic surface  of grade  $\geq 2$}): 

$S = (E_1 \times C_2) / G$, where $ E_1, C_2$ are elliptic curves, 
and $G$ acts  via an action of product type, such that $G$ acts
on $ E_1$ via translations, and faithfully on $C_2$ with  $ C_2 / G \cong  \PP^1$.

In this case all the fibres of the Albanese map are isomorphic to $C_2$, $P_{12}(S) = 1$, and  $S$ admits also an elliptic fibration
 $\psi : S \ra C_2 / G \cong \PP^1$.
 \smallskip

In the other cases (2.0,p),  (2.1,p), for $p \geq 1$, (2.1,0) with $g_2 \geq 2$,
 $S$ is {\bf isogenous to a higher genus elliptic product},  this means that 
  $C_2$ has genus $g_2 \geq 2$. Here $S$ is properly elliptic  and $P_{12}(S) \geq  2$.
  
   \medskip

The  cases are distinguished mainly by the geometric genus   $p_g(S)=p$.

In the  case of the torus and of the hyperelliptic surfaces  $K_S$ is numerically equivalent to zero, whereas in the other cases 
$K_S$ is not numerically equivalent to zero.

Moreover, the  three cases are also distinguished (notice that $c_1(S)$ is the class of the divisor $ - K_S$)  by 
\begin{itemize}
\item
$K_S = 0 \in Pic (S)$  for the case 1) of a complex torus, 
\item
$c_1(S)=0 \in H^2(S, \ZZ)$ but $K_S \neq 0 \in Pic (S)$
in the case of properly hyperelliptic surfaces,
\item
$c_1(S)\neq 0 \in H^2(S, \QQ)$ in the other cases where $S$ is isogenous to a higher genus elliptic product.
\end{itemize}

\end{theorem}
\begin{proof}
The part which is not contained in the cited sources  (\cite{bea}, Chapter 6, Theorem VI.13,    Theorem 1.4 of \cite{fcbl},  the crucial Theorem of \cite{time}  and
 \cite{haoschreiederbmy}) concerns how to distinguish the several cases.

Everything is straightforward, except the assertion that for hyperelliptic surfaces (case (2.1,0) with $g_1=1$)) the first 
integral Chern class is zero: but this is a special case of our result on Bagnera de Franchis manifolds, Theorem 3.1 of \cite{part1}.

\end{proof}
The following can be readily checked:

\begin{proposition}\label{autom}
In case  1) $A$ acts transitively and freely on $A$, so $Aut^0(S)$ has dimension 2.

In case  (2.1,p), $Aut^0(S)$ has dimension 1: $C_1= E_1$ acts on $S$,  transitively on the orbit closures, but with stabilizer 
$  H \subset G \subset E_1$ 
for the classes of points $ (x,y) \in E_1 \times  C_2$ such that $ H y = y$.

Finally, in case (2.0,p ), $Aut^0(S)$ has dimension 0 :  there is no action of $C_1$ on $S$, but the general fibres of the Albanese map $ \al : S \ra C_2/G$
are  isomorphic to $C_1$ (a finite number shall only be isogenous to $C_1$).
\end{proposition}

With a weaker notion of minimality we have a counterexample to the previous Theorem \ref{crucial},
as we shall now see.

\begin{proposition}\label{min}

If the smooth surface satisfies $c_1^2= c_2=0$, that is  $K_S^2= e(S)=0$, then $S$ is minimal.
\end{proposition}

\begin{proof}

By the Noether's formula,  from $ K_S^2=0$ and $e(S) =  c_2 (S) = 0$  follows $\chi(\hol_S)=0$.

And $\chi(\hol_S)$ is a birational invariant.

If $S$ is not minimal, then $S$ is the blow-up of a minimal surface $S'$ with $e(S') <0$. By Castelnuovo's theorem
$S'$ is a ruled surface with $q(S') \geq 2$. But then $0 > \chi(\hol_{S'}) = \chi(\hol_S) =0$, a contradiction.

\end{proof}

\begin{corollary}\label{non-min}
Consider  the minimal surfaces $S$ with Chern numbers $c_1^2= c_2=0$, that is with $K_S^2= e(S)=0$.

If $K_S$ is nef, $S$
is  isogenous to a product $Y \times A$, with a torus $A$ of dimension $\geq 1$.

If $K_S$ is not nef, then $S$ is a $\PP^1$-bundle over an elliptic curve.

\end{corollary}

\begin{proof}
By Theorem \ref{crucial} there remains only to consider the case where $K_S$ is not nef, hence $S$ is ruled.
Therefore, since the case of $S= \PP^2$ yields $e(S)=3$, we must have a $\PP^1$-bundle over a curve $C$.
In this case, since $0 = e(S)= 4 (1 - g(C))$, we get that $C$ has genus $g(C)= 1$.

\end{proof}

\begin{remark}\label{action}
 If we take  an elliptic curve $B$, and  a vector bundle $V$ of rank 2 on $B$,
say $ V = L \oplus M$, where $ deg (L) = deg (M)$,  then $\PP(V)$ is minimal, but the group of automorphisms of $\PP(V)$
consist of $\CC^*$ for general choice of $L, M$ (see \cite{maruyama}, also Theorem 7.3 of \cite{catliu} ).

Hence the action is not transitive on the orbit closures, because the two sections of $\PP(V)$ are left invariant by the
automorphism group.
\end{remark}

For completeness we show that:

\begin{proposition}
Surfaces in the class (2.0,0) do exist.
\end{proposition}

\begin{proof}
Let $G := (\ZZ/2)^3$, and make it first act on an elliptic curve $C_1$ as the group of transformations
$$ z \mapsto \pm z + \eta, \ 2 \eta = 0, $$
so that $G$ has generators $\eta_1, \eta_2, \e$, where $\e (z)=-z$.

To get a second action on $C_2$ such that $C_2 / G= : E_2$ is an elliptic curve, we take $E_2$
to be an elliptic curve, $\sB = \{ x_1, x_2\}$ a branch set, so that
$$ \pi_1 (E_2 \setminus \sB) = \langle \al, \be, \ga_1, \ga_2| \ga_1 \ga_2= [ \al, \be]\rangle,$$
hence $H_1 (E_2 \setminus \sB) = \ZZ \al \oplus  \ZZ \be  \oplus \ZZ \ga_1$.

Define $\mu : H_1 (E_2 \setminus \sB)  \ra G$ by:
$$ \mu (\al) = \e, \ \mu(\be) = \eta_1, \ \mu(\ga_1) = \eta_2 \Rightarrow \mu(\ga_2) = \eta_2.$$

We want to prove that the product action of $G$ on $C_1 \times C_2$ is free.

To this purpose we observe that $\eta_2$ has eight fixed points on $C_2$, and, since $C_2$ has genus $5$,
$E'_2 : = C_2 / \langle \eta_2 \rangle$ is an elliptic curve, so that $E'_2 \ra E_2$ is \'etale, that is, 
$ G / \langle \eta_2 \rangle$ acts freely on $E'_2$. The conclusion is that the only element acting on $C_2$
with fixed points is $\eta_2$; since $\eta_2$ atcs freely on $C_1$, the product action is free.

\end{proof}

We end this section observing that if $S$ is 
 a minimal surface with  $e(S)=0$, then  either $S$ is a $\PP^1$ bundle over an elliptic curve, 
 or $K_S$
 is nef. In the latter case it must be $K_S^2 =0$, since for $K_S^2 < 0$ $S$ is ruled, hence $K_S$ is not nef,
  and for $K_S^2 > 0$ $S$ is 
 of general type, and then $e(S) > 0$. 
 
 The following  is  the  correction  of  the theorem of Dantoni \cite{dantoni};
 it shows that if $S$ is not a torus then the  surface is `elliptic' only  in the weak sense
 that it has a rational map with  fibres elliptic curves,
 and not in the strong sense   that a torus of
 dimension at least 1   acts   on $S$,  as we saw in Proposition \ref{autom} 
 and in the previous remark \ref{action} (it was also stated as Theorem \ref{1-dantoni}
 in the Introduction with only slightly different wording).

 \begin{theorem}\label{dantoni}
 A  surface $S$ with  $e(S)=0$ is either the blow up of a $\PP^1$- bundle over a curve
 of genus at least 2, or it is minimal, and then it is either a complex torus, or 
 a hyperelliptic surface, or it  is isogenous to a higher genus elliptic product,
 or it is a $\PP^1$-bundle over an elliptic curve.
 
 In the last three cases $S$ contains  
 a 1-dimensional family of isomorphic elliptic curves whose union is dense in $S$.
 
 If the canonical divisor is numerically trivial, then $S$ is either a complex torus, or 
 a hyperelliptic surface.

 \end{theorem}

 \begin{proof}
 
 The statement follows immediately from:  Theorem \ref{crucial},  Theorem  \ref{detail},  Proposition \ref{autom}, 
 Corollary \ref{non-min} and from the statement and the proof of Proposition \ref{min}.
 
 \end{proof}

 \section{Pseudo Abelian  Varieties and   Varieties Isogenous to a $k$-Torus-Product}
 
 The first aim of this section is to show how to derive from Roth's original definition
 of the Pseudo Abelian Varieties the modern definition which we adopt in this article;
 and how Roth's definition leads, in the case of compact K\"ahler Manifolds,
 to the notion of Seifert fibrations.
 
 Later on we dwell on other notions, for instance we give a new definition, of suspension over a parallelizable Manifold.
 
 Leonard Roth \cite{rothPAV},  \cite{rothPAV2} defined the Pseudo-Abelian varieties as follows:
 
 \begin{definition}
 (Roth)
 
 A  smooth projective variety $X$ of dimension $n$ is said to be Pseudo Abelian of order $k$ if
 
 i) $Aut^0(X)$ contains a complex torus $T$, of maximal dimension $=k$, with the property  that its orbits are all of dimension $k$.
 \end{definition}
 
 Recall in fact the following theorem by Fujiki and Lieberman \cite{fujiki}, \cite{lieberman}
 
 \begin{theorem}\label{f-l}
 If $X$ is a compact K\"ahler manifold (or more generally in the Fujiki class)  then there is an exact sequence of groups 
 $$ 1 \ra L \ra Aut^0(X) \ra T_X \ra 1$$
 where $L$ is a linear algebraic group, and $T_X$ is a complex torus. 
 
 The Lie Algebra of $Aut^0(X)$ is $\mathfrak A := H^0(\Theta_X)$, while the Lie Algebra
 of $L$ is the space $\mathfrak L$ of vector fields which admit zeros
 ($\mathfrak L$ is an ideal in the Lie Algebra $\mathfrak A$).
 
 If $X$ is not uniruled, then $L= \{1\}$.
 \end{theorem}
 
 We  explain now,  sketching an argument of proof,   how to derive from Roth's definition a simpler one.
 
 We have an action $ T \times X \ra X$, and we denote by $T_x $ the orbit of $x$, i.e. the  image of $T \times \{x\}$.
  $T_x$ is smooth because no subvariety $Z$ of $T$ can be  exceptional, the normal bundle being spanned by global sections.
 Recall also that we are assuming that $X$ is smooth.
 
 Hence the orbits $T_x$ of $T$ give a subvariety $\sV$ of dimension $n-k$ in the Hilbert scheme $\sH$ (Douady space) of $X$,
 and the restriction of the universal family to $\sV$, $ \phi: \sU \ra \sV$, yields $\sU$ which
  maps isomorphically to $X$. In fact, $ dim (\sU) = n = dim (X)$, and the identity of $X$ factors as:
 $$ X \ra \sU  \ra \sV \times X \ra X,$$
 where $ x \in X \mapsto ( T_x, x)$.

Hence we may write:  $ \phi: X \ra \sV$, and since the action of $T$, 
 $$ a : T \times X \ra X$$ commutes with the projection over $\sV$, 
and is effective, then  the general fibre of $\phi$ is isomorphic to $T$.

 The other fibres are instead of the form $T/G'$, where $G'$ is a finite subgroup of $T$.
 Since the general fibres are isomorphic,   we have a holomorphic bundle over an open set $\sV'$ of $\sV$
 (for instance, as a consequence of  Kuranishi's theorem (see \cite{kuranishi}, \cite{wavrik}). 
 
 Because of the action of $T$ on the fibres the monodromy of the bundle centralizes the group of translations 
 hence the monodromy transformations consist of translations, and we have a principal bundle.
 
 If   $X$ is projective, then the monodromy is finite, hence we get a finite group $G \subset T$.
 
 The fibration is isotrivial, hence 
there exists a finite Galois base change $f : Y \ra \sV$ with group $G$ such that the fibre product 
is birational to a product
 $$ Y \times_{\sV} X  \sim Y \times T.$$
 
 At each point of $\sV$ the local monodromy $G'$ is a subgroup of $G$, hence the fibre product
 $ Y \times_{\sV} X$ yields an unramified covering of $X$ ( since $G'$ yields an unramified covering of the 
 corresponding fibre).
 
Hence the fibre product  $ Y \times_{\sV} X$ is smooth,  since $X$ is smooth and we have an unramified covering.

$ Y \times_{\sV} X$ is smooth,   is birational to $Y \times T$,
and all the fibres are isomorphic to $T$:
therefore,  $ Y \times_{\sV} X \cong Y \times T,$
compatibly with the projection onto $Y$.

 This motivates an equivalent definition, and some related definitions:
 
  \begin{definition}
 
(I)  A complex manifold  $X$ of dimension $n$ is said to be a k- Pseudo-Torus (or Pseudo-Torus of order $k$) if
 there is a torus $T$ of dimension $k$, a manifold $Y$, and a finite Abelian group $G$ 
 acting on $T$ faithfully via translations, acting faithfully on $Y$, such that the quotient
 of the product action is isomorphic to $X$:
 $$ X = (Y \times T)/G.$$
 
 (II) If moreover $X$ is projective, we shall say that $X$ is a Pseudo-Abelian Variety of order $k$ in the strong sense.
 
 (III) A compact K\"ahler Manifold $X$ is said to be  Seifert fibred, cf. \cite{lieberman}, if there is
 a finite abelian unramified covering $Z \ra X$ with Abelian Galois group $G$ (hence $X = Z/G$)
 such  that $Z$ is a principal bundle $ Z \ra Y$, with fibre a (positive dimensional) complex torus $T$,
 and where the action of $G$ commutes with the action of $T$ on $Z$.
 
 (IV) A compact complex manifold $X$ is called (\cite{amoros}) a suspension over a complex torus $T = \CC^k / \Lam$
 if there is a compact complex manifold $Y$ and a homomorphism $\rho : \Lam \ra Aut(Y)$
 such that
 $$ X =( \CC^k \times Y) / \Lam, \la (z,y) : = ( z + \la , \rho(\la) (y)) .$$

 \end{definition}
 
 \begin{remark}
(a)  The reason for definition (III) is that if $X$ is only a cKM, and it has an action of a  complex torus $T$, with all the orbits
tori of the same dimension,  then the global monodromy is not necessarily finite.
 But the local monodromies around the points corresponding to multiple fibres are finite subgroups of $T$, and since we have a finite number of them, the local monodromies generate a finite subgroup $G$ of $T$. Associated to this subgroup there is
 a  covering $Y \ra \sV$ which is a principal $T$-bundle.
 
 (b) In the case of definition (IV) of a suspension over a torus, $\CC^k$ acts on $X$ via translations on the first summand, but in general there 
 is no complex torus acting on $X$.
 
 Indeed the  $\CC^k$-orbits are the projections $\pi (\CC^k \times \{ y\})$, which are isomorphic to
 $  \CC^k / Stab_y$, hence they do not need be compact.
 
 The suspension over a torus is a Pseudo-torus if and only if the subgroup $ Im (\rho) \subset Aut(Y)$
 is finite. Equivalently, if and only if all orbits are tori (which amounts to $Stab_y$ being a subgroup of
 finite index in $\Lam$ for all $ y \in Y$): since then  $ Ker (\rho)$ has finite index.
 
 In particular the suspension is Pseudo-Abelian if and only if it is a Pseudo-torus.
  
 (c) Observe  that for such a suspension  we have a splitting of the tangent bundle
 
 $$\Theta_X \cong \hol_X^k \oplus \sF,$$ 
 hence also of the cotangent bundle.

 In \cite{amoros}  Example 2.4, page 1005, it is observed that a suspension $X$ as above is K\"ahler
 if and only if $Y$ is K\"ahler and a finite index subgroup  $\Lam ' \subset \Lam$ maps to $Aut^0(Y)$.
 
 In this case, taking $G : = \Lam / \Lam'$,
 there is a Galois covering $Z$ of $X$ with group $G$.  If moreover $Aut^0(Y)$ is trivial, then $Z = T' \times Y$
 and we have a pseudo-torus.
 
 If  $Aut^0(Y)$ is non-trivial, one cannot conclude that $Z$ is a principal torus bundle, since the action of $\Lam'$ on $Y$
 need not be properly discontinuous.
 
 (d) If $X$ is Seifert fibred, then $X = Z/G$, and we have an exact sequence corresponding to the principal bundle $ Z \ra Y$:
 $$ 0 \ra \hol_Z^k \ra \Theta_Z \ra \sF_Z \ra 0.$$
 Since the action of $G$ commutes with the action of $T$, the exact sequence descends to $X$, and we have
 $$ 0 \ra \hol_X^k \ra \Theta_X \ra  \sF  \ra 0.$$

 \end{remark}

 As we shall soon see, the above notions of Seifert-fibred manifold, or of pseudo-torus,  and of suspension over a torus
 are not general enough  in order to deal with 
 manifolds with Chern classes trivial in real cohomology, hence we give another definition (important for  the case of projective varieties):

   \begin{definition}
 
   Recall the definition of isogeny given in Remark \ref{isogeny}: a compact complex Manifold  $X$ of dimension $n$ is said to be a 
 \begin{itemize} 
 \item
 (i) {\bf  Manifold Isogenous to a k-Torus product }
 (MITP of order $k$ ), if
 $X$ is isogenous to a product $ Y \times T$, where 
 $T$  is a torus  of dimension $k >0$,    $Y$ is a compact complex manifold. 
 If  $k$ is maximal with this property we say that $X$ is a maximal MITP of order $k$.
 \item
 (ii) {\bf  (Manifold) Strongly Isogenous to a k-Torus product } (SITP of order $k$ ),
 if there  is a torus $T$ of dimension $k >0$,  a compact complex  manifold $Y$,  and a finite group $G$  acting on $T$, and $Y$, such that $G$ 
 acts freely on the product $ Y \times T$ via the induced diagonal action 
 ($g(t,y) : = (g(t), g(y))$ and moreover:
  $$ X = (Y \times T)/G.$$
 If  $k$ is maximal with this property we say that $X$ is a maximal SITP of order $k$.
 \end{itemize}
 
 If $X$ is projective, we shall say that it is a  Variety Isogenous to a k-Torus product.

 \end{definition}
 
 \begin{proposition}
  If $X$ is not uniruled, and $X$ is a compact K\"ahler Manifold, the two notions 
  (i) and (ii) are equivalent in the maximal case.

\end{proposition}
 
 \begin{proof}
 
 Clearly (ii) implies (i), hence it suffices to show that (i) implies (ii),
 and we shall do it by induction on the number $h$ of finite unramified maps which
 make $X$ isogenous to such a product $ Y \times T$.
 
 \smallskip
 
 {\bf Step (a(1))} For $h=1$, assume that there is a finite unramified map $ f : X \ra Y \times T$.
 
 Then, taking the Galois closure $Z$, and letting $G$ 
 be the Galois group, we see that $Z$ is associated to an epimorphism
 $\psi :  \pi_1(Y) \times \pi_1(T) \ra G$, and we denote 
 $G_1: = \psi (\pi_1(Y)), \ G_2: = \psi (\pi_1(T))$. 
 
Then $ G $ is a quotient of $G_1 \times G_2$ by the normal subgroup $G_1 \cap G_2 < G$.

To the epimorphism
 $\phi :  \pi_1(Y) \times \pi_1(T) \ra (G_1 \times G_2)$ corresponds a product Manifold
 $Y_1 \times T_1$, and clearly $ Z = (Y_1 \times T_1) / G_1 \cap G_2$.
  We show now that  $X$ is   strongly isogenous to a product. Indeed  $G_1 \times G_2$  acts on $Y_1 \times T_1$ diagonally
 (and on $T_1$ via translations). Hence $X$, which is     a quotient of  $(Y_1 \times T_1)$
 by another subgroup  $G'$ of $G_1 \times G_2$ containing $G_1 \times G_2$,   is strongly isogenous to a product.
 
 {\bf Step (b(1))} For $h=1$, assume that there is a finite unramified map $ f :  (Y \times T) \ra X$. If $f$ is not Galois, take the Galois closure $Z$,  and apply Step (a(1)) to $Z$: using the same notation for the groups involved,  we find a Galois cover 
 $(Y_1 \times T_1 ) \ra (Y \times T)$ with group $G_1 \times G_2$, and $Z$ is a quotient of $(Y_1 \times T_1 )$ by a subgroup $G'$
 of  $G_1 \times G_2$.
 
Denote by $\Ga$
 the Galois group of $Z \ra X$, and by $G$ the subgroup which is the Galois group of  
$Z \ra (Y \times T)$. 
 
 We claim that the unramified cover $(Y_1 \times T_1 ) \ra X$ is Galois; to establish this claim it suffices to show that the elements $ \ga \in \Ga$ admit lifts to $(Y_1 \times T_1 )$,  because then the group of such  lifts $\Ga'$ admits and exact sequence $ 1 \ra G' \ra \Ga' \ra \Ga \ra 1$, hence $ X = (Y_1 \times T_1 )  / \Ga'$.
 
 Now, $Z$ has a split tangent bundle $ \Theta_Z = \sF \oplus \hol_Z^k$,
 and since $Z$ is not uniruled, it follows from  the Theorem of Fujiki and Lieberman, Theorem \ref{f-l}, 
 and by the maximality of $k$, that $H^0(\Theta_Z) = H^0(\hol_Z^k)$.
 
 Hence the splitting is uniquely determined, and the group $T_2$ maps onto 
 $Aut^0(Z)$. The fibration $ Z \ra (Z / T_2)$ is then isotrivial, and  $(Y_1 \times T_1 )$ is the canonical product pull back under $Y_1 \ra (Z / T_2)$.
 
 Therefore the lifting property holds true.
 
 \bigskip
 
  We assume now by induction that the Theorem is proven for $X_2$ if at most $h-1$ unramified maps are needed
 to make $X_2$   isogenous to a product, and consider $X$ such that
 $h$ such maps are needed. Then there is such an $X_2$  such that there is an \'etale (finite unramified map)
 between $X, X_2$. We have again two cases, the first being where  there is a finite unramified map $ f : X  \ra X_2$,
 the second where there is a finite unramified map $ f : X_2  \ra X$. We proceed to analyse these two cases.

{\bf Step (a(h))} Assume  that $ X_2 = (Y_2 \times T_2)/G_2$, as in (ii), and that there is a finite unramified map $ f : X  \ra X_2$.

Then, taking a component $X'$ of the fibre product $ X \times _{X_2}  (Y_2 \times T_2)$,
and applying  (a(1)) we get that $X'$, hence also $X$, admits $(Y_1 \times T_1)$
as an unramified covering, hence we can conclude because of Step (b(1)).

{\bf Step (b(h))} Assume then that $ X_2 = (Y_2 \times T_2)/G_2$, as in (ii), and that there is a finite unramified map $ f : X_2  \ra X$. 

Then taking the composition we get an unramified map $(Y_2 \times T_2) \ra X$,
and again we are done by Step (b(1)).

 \end{proof}

 The following  are easy important properties of such Manifolds and Varieties Isogenous to a Torus product.
 
 \begin{proposition}
  Let the complex Manifold $X$ of dimension $n$  be  Isogenous to a k-Torus product $Y \times T$
  where
  $ dim (T) = k > 0$.
  
  Then the integral Chern classes $c_i(  Y \times T) \in H^* (Y \times T, \ZZ)$ vanish 
  for $ i \geq n-k +1 $, and the rational Chern classes $c_i(  X) \in H^* (X, \QQ)$ vanish 
  for $ i \geq n-k + 1$.

  Moreover all the Chern numbers of $X$ vanish.
  
  If moreover $X$ is (respectively is isogenous to) a pseudo-torus, or more generally the suspension over a torus,
  or is Seifert fibred,
  then all the integral Chern classes $c_i(  X) \in H^* (X, \ZZ)$ vanish 
  for $ i \geq n-k +1 $ (respectively the corresponding rational Chern classes vanish).

 \end{proposition} 
 
 \begin{proof}
 Since the tangent bundle of $Y \times T$ is the  direct sum of the pull back of the
 tangent bundle of $Y$ with the pull back of the
 tangent bundle of $T$ which is trivial, the total Chern class of $Y \times T$ 
 is the pull back of  the total Chern class of $Y $, and this proves the first assertion.
 
The second assertion follows since the Chern classes of $X$ pull back to the Chern classes 
of the product $Y \times T$, and $H^*(X, \QQ) =  H^* (Y \times T, \QQ)^G$.

The third  assertion follows since any isobaric polynomial of weight $n$ in the Chern classes of $X$
is a rational multiple of the same  isobaric polynomial of weight $n$ in the Chern classes of $Y$,
and we are assuming $k >0$.

The last assertion follows since the tangent bundle of $X$ either splits as $\Theta_X = \hol_X^k \oplus \sF$,
or has a bundle  exact sequence $$ 0 \ra \hol_X^k \ra \Theta_X  \ra  \sF \ra 0.$$
 
 \end{proof}
 
 \section{On the non K\"ahler case}
 In the realm of compact complex manifolds, the Manifolds $X$ with a holomorphically trivial tangent bundle $\Theta_X \cong \hol_X^n$ are called the {\bf parallelizable Manifolds}.

 By a Theorem of Wang \cite{wang} every parallelizable Manifold
 is a quotient $X =  \sG / \Lam$, where $\sG$ is a  simply connected complex Lie group 
 and $\Lam$ is a cocompact discrete subgroup.
 
 Since every Manifold isogenous to a parallelizable Manifold has trivial 
 real Chern classes, it would be interesting to study this class in more detail
 (in the K\"ahler case the parallelizable Manifolds are the complex tori,
 and the Manifolds isogenous to a torus are the Hyperelliptic Manifolds).
 
 \begin{proposition}
 Every Manifold $Y$  isogenous to a parallelizable Manifold
$X' =  \sG / \Lam'$ is obtained from another parallelizable Manifold $X =  \sG / \Lam$
dividing by the action of a finite group $G$ acting freely on $X$.

$G$ is a group of affine transformations of $X$, 
of the form $f(x) = g F(x)$ where $ g \in \sG$ and $F$ has a fixed point $x_0$ 
and has derivative $DF$ determined by the value of the derivative at $x_0$.

The quotient Manifold $Y : = X / G$ will be called a {\bf 
twisted hyperelliptic Manifold}.

 \end{proposition}
 
 \begin{proof}
 If $Y$ is isogenous to $X'$, then $Y$ has also $\sG$ as universal cover,
 as it is trivial to see.
 
 Hence we can write $Y = \sG / \Ga$, where $\Ga$ acts freely on $\sG$
 and with compact quotient.
 
 Observe that  any isogeny amounts, up to isomorphism, to a finite index containment
 between fundamental groups, viewed as discrete groups of automorphisms of $\sG$ (since $\sG$ is simply connected).
 
 It is a general fact that the intersection $H\cap K$ of two finite index subgroups $H, K$ of
 a group $\Ga'$ is again of finite index in $\Ga'$  (it is the stabilizer of the orbit of $ H \times K$ on $\Ga' / H \times \Ga' / K$
  %\footnote{see  for instance: stack Overflow}).
 
 Hence, if  we consider subgroups $\Ga_i < \Ga', \ i=1, \dots, 5,$ such that  $\Ga_2$ is of finite index in $\Ga_1$ and $\Ga_3$, 
and $\Ga_4$ is of finite index in $\Ga_3$ and $\Ga_5$,
 then $\Ga_2 \cap \Ga_4$ is of finite index in $\Ga_3$, hence also 
 in  $\Ga_2$ , $\Ga_4$, $\Ga_1$ and $\Ga_5$.
 
 Therefore, since we can always assume  that  the number of direct relations  
 yielding the isogeny is even (introducing an equality step),
 we can reduce to the case where $ \Ga$ and $\Lam'$ contain a finite index subgroup in
 both of them, call it $\Lam''$. By replacing $ X'' = \sG / \Lam''$ by the Galois closure
 of $X'' \ra Y$, we obtain $Y  = X / G$ as desired.
 
 Now, if $f \in G$, there exists $g \in \sG$ such that $ F : = g^{-1} f$ has a fixed point 
 $x_0 : = \Lam$ on $X$. The derivative of $F$ is a holomorphic section of the trivial bundle $ DF \in H^0( End (\hol_X^n))$
 gotten by the action on the space of left invariant vector fields: since $X$
 is compact, $DF$ is constant, hence $ f (x) = g F(x)$ is affine.
 
 \end{proof}
 
 Of course each  such automorphism $F$ lifts to an automorphism $\tilde{F}$ 
 of $\sG$ which must have the property  of fixing the identity and of 
 sending $\Lam$ to $\Lam$, since we must have
 $$ \tilde{F} (g \la ) =  \tilde{F} (g ) \pi_1(F) (\la) \Rightarrow \tilde{F} (\la ) = \pi_1(F) (\la),$$
 this is exactly as it happens for automorphisms of complex tori.
 
 We give now a simple example of such Manifolds.
 
 \begin{proposition}
 Consider the Iwasawa Manifold  $X =  \sG / \Lam$, where $\sG$ is the complex Lie group of 
 unipotent $ 3 \times 3$ complex matrices, that is, upper triangular matrices and with diagonal entries equal to $1$.
 
 And let $\Lam$ be the subgroup of matrices with upper-diagonal entries belonging to the ring $\ZZ[i]$
 of Gaussian integers.
 
 Then we define the following automorphism of $X$, such that 
 $$ r [(z_{12}, z_{21},z_{13})] : = [(i z_{12},  z_{21} + 1/4 , iz_{13} )].$$

  Then $r$ has order $4$, and generates a group $G \cong \ZZ/4$ acting freely on $X$.
  The quotient is a twisted Hyperelliptic Manifold with $H^0 (\Omega^1_Y) $
  of dimension $1$ and generated by a closed holomorphic 1-form..
 \end{proposition}
 
 \begin{proof}
 The automorphism is well defined, since the product $z \la$ has coordinates
 $$ (z_{12} + \la_{12}, z_{21} + \la_{21}, z_{13} +\la_{13} + z_{12} \la_{21}) ,$$
 hence  we infer $ [ r (z)] = [ r (z \la)] $ since 
 
 $$ [ r (z \la)] = [ (i z_{12} + i \la_{12}, z_{21} + \la_{21} + 1/4, i z_{13} + i \la_{13} + i z_{12} \la_{21})]=$$
 $$ =
 [ (i z_{12}, z_{21} + \la_{21} + 1/4, i z_{13}  + i z_{12} \la_{21})] = [ r(z) \la'], $$
 where $  \la' = (0,\la_{21}, 0) $.
 
 That the action is free follows since for the second coordinate 
 $$ z_{21} \mapsto z_{21} + 1/4.$$
 
 It is known that  $H^0 (\Omega^1_X)$ is generated by 
 $$ \om_1 : = d z_{12} , \om_2 : = d z_{21},\om_3 := d z_{13} -  z_{21} d z_{12} ,$$
 which are mapped by $r$  to  $ i \om_1  , \om_2 , i \om_3 - ( i/4) \om_1 $.
 
 Hence $H^0 (\Omega^1_X)^G$ is generated by $\om_1 $.
 \end{proof} 
 
 We leave aside here the general discussion of the structure of twisted 
 Hyperelliptic Manifolds, but we raise a question.
 
 \begin{question}
 The  twisted 
 Hyperelliptic Manifolds, being isogenous to parallelizable Manifolds,
  are examples of compact complex manifolds with all the real Chern classes 
  $c_{i, \RR}(X) =0 , \forall i$.
  
  Are there other examples?
 \end{question}

\bigskip
 
 We can also define  another general class of Manifolds,
 which we call Suspensions over a parallelizable Manifold.

 \begin{definition}
 
 (S-P-M) A compact complex manifold $X$ is called  a {\bf suspension over a 
 parallelizable Manifold} $Z =  \sG / \Lam$ 
if there is a compact complex manifold $Y$ and a homomorphism $\rho : \Lam \ra Aut(Y)$
 such that
 $$ X =( \sG \times Y) / \Lam, \la (z,y) : = ( z  \la ,  y \rho(\la)) .$$

 \end{definition}
 
 \begin{proposition}
 If $X$ is   a suspension over a 
 parallelizable Manifold $Z$, then we have a splitting of the tangent bundle 
 $$\Theta_X = \hol_X^k \oplus \sF,$$
 where $k = dim (Z)$.
 \end{proposition}
 \begin{proof}
 The splitting  follows right away since $\Lam$ has a product action.
 
 Moreover, the first summand is trivial because  $Z =  \sG / \Lam$ is
 parallelizable and the first summand is just the pull-back of the tangent bundle 
 $\Theta_Z$ for the projection $ X \ra Z$, with fibres isomorphic to $Y$.
 
 \end{proof}
 
 \begin{remark}
 In the case where $\rho$ has finite image, we get that, $\Lam'$ being 
 defined as $\Lam' : = ker (\rho)$,
 then $X$ is isogenous to $Z' \times Y$, where $Z' : = \sG / \Lam'$.
 
 And we get then triviality of the Chern numbers of $X$.
 \end{remark}
 
In the next section we shall show how, assuming that $X$ is  K\"ahler ,  the occurrence of such a splitting 
$\Theta_X = \hol_X^k \oplus \sF$ is understood as the structure of a suspension over a torus.

It is interesting  to see whether a similar characterization of suspensions over  a parallelizable Manifold  holds also for compact
complex manifolds. 

For this purpose one has to  take $k$ maximal, and the first key point would be  to see whether  
$ H^0 ( \hol_X^k) \subset H^0(\Theta_X)$ is a Lie subalgebra $ \mathfrak G$: then we would have an action
on $X$  of the 
simply connected Lie group $\sG$ associated to $ \mathfrak G$.
 
 \section{Partially framed and co-framed manifolds}
 
 We begin with a simple definition yielding complex Manifolds with all
 the last $k$  integral Chern classes equal to zero.
 
  \begin{definition}
 
 A complex manifold  $X$ of dimension $n$ is said to be $k$-tangentially framed,
 or simply $k$-framed,
 if the holomorphic tangent bundle $\Theta_X$ admits a trivial subbundle  $ \cong \hol_X^k$,
 and where $k > 0 $ is maximal.
 
  A complex manifold  $X$ of dimension $n$ is said to be $k$-cotangentially framed,
 or simply $k$-co-framed,
 if the holomorphic cotangent bundle $\Omega^1_X$ admits a trivial subbundle  $ \cong \hol_X^k$,
 and where $k > 0 $ is maximal.

 \end{definition}
 
 In the rest of the section we shall use the results of \cite{lieberman}, \cite{fujiki}, \cite{amoros},
 often for simplicity we might refer to the exposition given in the last paper.

  \begin{proposition}\label{nozero}

  Assume that $X$ is a $k$-coframed compact K\"ahler manifold  and let  
  $W \subset h^0(\Omega^1_X) $  be the corresponding   maximal  vector subspace consisting entirely 
  of nowhere vanishing holomorphic 1-forms (of dimension  $k \geq 1$). Then $W$ determines 
  an analytically  integrable foliation with trivial normal bundle.
  
  i) If moreover the subspace $W$ corresponds to a quotient torus $T'$ of $Alb(X) = H^0(\Omega^1_X)^{\vee} / (H_1(X, \ZZ)/ Tors)$,  
  then  the foliation is algebraically integrable, consisting of the fibres of $ \Psi : X \ra T'$,
which is a differentiable fibre bundle. 
  
ii) If we have a splitting $\Omega^1_X \cong \hol_X^k \oplus \sF^{\vee}$, then $\Psi$ is a holomorphic fibre bundle with fibre $Y$.
 
 iii) If moreover $Y$ has   finite automorphism group,
  then $X$ is a   k-Pseudo-Torus product $X = (Y \times T)/G$
  with 
  $ dim (T) = k > 0$.

 \end{proposition} 
 
 \begin{proof}

  A subspace $W$ which is maximal with the property that all forms $\omega \in W \setminus \{0\}$
 are nowhere vanishing, has a basis  $\omega_1, \dots, \omega_k$
such that  the forms $\om_j$ are linearly independent   at each point,
hence they generate a trivial rank $k$ subbundle of $\Omega^1_X$.

The associated foliation is  analytically integrable, because  
  $X$ is a cKM, and holomorphic 1-forms are closed, hence the distribution induced by 
 $W$ is integrable, and spans a trivial subbundle of the cotangent bundle. 
 
 The foliation on $X$ is induced by a foliation on $Alb(X)$, corresponding to the annihilator of $W$.
 This foliation is algebraically integrable if, as we assume, it corresponds to the projection onto a quotient torus $T'$.
 
 The composed map $ \Psi : X \ra T'$ has fibres of dimension $k$, which are therefore union of leaves.
 Since the fibres are smooth, if they are not connected, we would get by the Stein factorization on unramified 
 covering of $T'$, which is again a quotient of $Alb(X)$ by the universal property of the Albanese variety.
 
 Hence we may assume that the fibres of $\Phi$ are connected, and we have a differentiable fibre bundle.
 
 The splitting $\Omega^1_X \cong \hol_X^k \oplus \sF^{\vee}$ guarantees that the Kodaira-Spencer map
 for the family is identically zero, hence by Kuranishi's theorem we have a holomorphic bundle.
 
 If this is a holomorphic bundle, with fibre $Y$, and  $Aut(Y)$ is finite, there is an unramified covering $T \ra T'$
 with group $G$ such that the pull back is a product. Therefore $ X = (T \times Y)/G$,
 where $G$ acts on $T$ via translations.
 
 \end{proof}

  \begin{theorem}\label{framed}
 Let $X$ be a compact K\"ahler  complex Manifold  of dimension $n$: then $X$ is      the suspension over a torus $T$
  with 
  $ dim (T) = k > 0$ if and only if there is a $k$-framing yielding a partial tangential splitting 
  $$\Theta_X \cong \hol_X^k \oplus \sF.$$

   A $k$-framing of $X$ yields the structure of a Seifert fibration on $X$ in the case where $h^0(\Theta_X) =k$.

  Moreover, a $k$-framed projective manifold $X$ is  a suspension over an Abelian Variety 
  and is indeed a Pseudo-Abelian variety if $K_X$ is nef.

 \end{theorem}
 
 \begin{proof}
  We have already seen that for a suspension over a torus we have such a partial tangential splitting.
 
 Conversely, recall that $H^0(\Theta_X)$ is the Lie algebra of the Lie group $Aut(X)$.
 
The Lie Algebra $H^0(\Theta_X) = : \mathfrak A_X$ contains the Lie ideal $\sH^1_X$ of the vector fields
admitting zeros, and there is (see \cite{amoros}, page 1002), a direct sum 
$$ \mathfrak A_X = \sH^1_X \oplus \sA,$$
where $\sA$ is a maximal Abelian subalgebra  consisting of   nowhere vanishing vector fields.

By Theorem 3.14 of \cite{lieberman}, $\sH^1_X$ is precisely the subspace of $\sH_X$ yielding the zero flow on $Alb(X)$.

  If $\sH^1_X =0$ then the trivial subbundle yields $k$ everywhere linearly independent vector fields, which, see \cite{fujiki}, \cite{lieberman},
 and also Theorem 1.2 of \cite{amoros}, 
 generate the action of a $k$-dimensional complex torus $T$ with smooth orbits $T_x$, quotients of $T$
 by a finite group $H_x$.
 
 Hence $X$ is Seifert fibred, as we have discussed earlier, cf. Theorem 4.9  by Lieberman \cite{lieberman}.
 
 If  we have  a tangential splitting $\Theta_X \cong \hol_X^k \oplus \sF$, we get a corresponding
 cotangent splitting, 
 $$\Omega^1_X \cong \hol_X^k \oplus \sF^{\vee}.$$
 It suffices, by the preceding proposition, to show that the coframing defines a subspace $W \subset H^0(\Omega^1_X)$
 corresponding to a quotient torus $T'$ of  $Alb(X)$.
 
 Since every automorphism of $X$ yields an affine action on $Alb(X)$, the action of the torus $T$, which spans
 $W^{\vee}$, yields a subtorus $A$ of $Alb(X)$. Then we define $ T' : = Alb(X) / A$, and we get $\Psi: X \ra T'$
 which   is a holomorphic bundle, with parallel transport given by the
 action of $T$: hence $X$ is the suspension over a torus.

Finally, if $X$ is projective, $T'$ is an Abelian variety. We defer the reader to Theorem 0.3 of \cite{amoros}
for the proof of the last assertions, see also \cite{lieberman1}.

 \end{proof}
 
 \begin{remark}
Theorem 0.3 of \cite{amoros} proves the following very  interesting result:  if $X$ is a compact K\"ahler manifold which
is $k$-framed, then $X$ admits a small deformation which is a suspension over a $k$-dimensional torus,
and which is a k-Pseudo-Torus if $Kod (X) \geq 0$.

If $X$ is projective and $k$-framed, as we saw, they show that $X$ is Pseudo-Abelian.

 \end{remark}
 
 \begin{theorem}\label{coframed}
 
  Assume that $X$ is a $k$-coframed compact K\"ahler manifold, such  that either
 
 (a) $q: = h^0(\Omega^1_X)=k$,
 
  or  more generally 
  
  (b) the framing is given by $f^* ( \Omega^1_A)$, where $f : X \ra A$ is holomorphic
  and $A$ is a complex torus of dimension $k$ ($f$ is the Albanese map $a_X$ in the first case).
  
  (i)  Then $f$    is a differentiable fibre bundle.
 
(ii)  If   $X$ is projective,  and $K_X$ is nef, then $X$ is a pseudo-Abelian variety provided one of the following assumptions is satisfied:

\begin{enumerate}
\item
$n-k \leq 3$,
\item
$K_X$ is pseudoample (there is a  positive integer $m$ such that  $|m K_X|$ is base point free),
\item
the fibres of $f$ have a pseudoample canonical system,
\item
$k \leq 5$.
\end{enumerate}

 \end{theorem}
 
 \begin{proof}
 The first assertion follows from i) of Proposition \ref{nozero}.
 
We prove now as a warm up   the  second assertion (ii) in the case where the fibres of $f$ have dimension $1$, 
hence they are  smooth curves of genus $g \geq 1$.
Since the fibration induces a holomorphic map  of the universal covering $\CC^k$ of $A$ to the Teichm\"uller space $\sT_g$,
 which is biholomorphic
to a bounded domain, this map is constant and we have a holomorphic fibre bundle.

Since $X$ is projective, the monodromy is finite, so there exists a finite unramified map $A' \ra A:= Alb(X)$ such that
the pull back is a product, hence $X$ is a Pseudo-Abelian variety.

Since $K_X$ restricts to the canonical bundle on the fibres, the fibres have nef canonical divisor,
and clearly (2) implies (3).
 
Also  (1) implies (3): if   the fibres  have dimension $\leq 3$,  then the abundance conjecture is true
by  results of Miyaoka and Kawamata, as explained  in \cite{flips}, that is, (3) holds.

We are reduced to 
proving the result under assumption (3), which implies that the fibres have a good minimal model.

We can apply then  Theorems 1.1, 1.2 and 1.3 of \cite{taji} implying that we have birational isotriviality, namely,  there exists  a  base change $\phi : W \ra A$ such that the pull-back 
$$\phi^* (X) : =  W \times_A X$$ of $f : X \ra A$
is birational to a product. Given a general smooth fibre $F$ of $f$, we can assume that $\phi^* (X)$ is birational
to $ W \times F$, and that this birational map induces  the identity on $F$. 

Hence there is a Zariski open set $U$ of $A$ such that the fibres over $U$ are isomorphic to $F$.

We apply now the  results of Wavrik\cite{wavrik} in order to conclude that, for each fibre $F_a$, the Kuranishi space $ Def(F_a)$ is a local moduli space,  and from this we shall conclude  that all the smooth fibres are isomorphic to $F$ and 
that we have a holomorphic bundle. 

The hypothesis needed is that the dimension of the space of holomorphic vector fields
$ h^0 (F_a, \Theta_{F_a})$ is constant.  Now, for $a$ in the Zariski open set $U$,  this dimension equals $k$, 
hence, since   $ h^0 (F_a, \Theta_{F_a})$ is   upper semicontinuous by standard results, 
 for each point $a$
this dimension can only be $k_a \geq k$.
Since the fibres $F_a$  have nef canonical divisor, they do not admit vanishing holomorphic vector fields,
see  Theorem  \ref{f-l},  hence their space of vector fields is the Lie algebra of a complex torus $Aut^0(F_a)$,
hence the fibres, being projective with nef canonical divisor, are pseudo-Abelian.
 Since on the fibre $F_a$ we have (see Theorem \ref{framed}) a bundle splitting 
$$\Omega^1_{F_a} = \sF' \oplus \hol_{F_a}^{k_a}$$
and $h^0(\Omega^1_{F_t})$ is constant for the fibres $F_t, \ t\in A$,  we see that this splitting extends to the nearby fibres $F_t$,
$ t \in U$:
if $k_a > k$, 
 the maximality of $k$ for $F_t$ would be contradicted.

Hence the  dimension $ h^0 (F_a, \Theta_{F_t})$ is constant for $t \in A$ and  we can apply Wavrik's result.

We have now a holomorphic map of a neighbourhood $\sU$ of $a \in A$ into $ Def(F_a)$, and this map is constant on a dense subset:
hence this holomorphic map is constant, and Wavrik's Theorem  shows that $f$ is a holomorphic bundle.

Since $X$ is projective, we conclude as above that  $X$ is a pseudo Abelian variety.

Under  assumption (4) we can instead apply  Theorem 1.1 of  \cite{taji21} to obtain birational isotriviality, and proceed similarly.

 \end{proof}

\begin{remark}

(I) Isotriviality of such a fibration  
was first  shown by Kovacs in the case where the fibres have ample canonical divisor
 \cite{kovacs}, and this result was extended in \cite{wei-wu} who proved birational isotriviality in  the more general case of fibres of general type.

\end{remark}

 \begin{example}\label{algcoframed}
 
 This example concerns $k$-coframed varieties.
 
 (1) A smooth fibration onto an Abelian variety need not be isotrivial, and not even a holomorphic bundle.
 
 Let in fact $A$ be an elliptic curve and assume that we have an embedding $ j : A \ra S$, where $S$ is a surface of general type,
 say with $K_S$ ample.
 
 Let $\Ga \subset A \times S$ be the graph of $j$, and take $ X $ to be the blow up of $A \times S$ with centre the smooth curve 
 $\Ga$.
 
 Then the differentiable fibre bundle $ f : X \ra A$ is not a holomorphic bundle, since $Aut(S)$ is finite, and
 the pairs $(S,P_1)$ and $(S, P_2)$ are not isomorphic for general $P_1, P_2 \in j(A)$.
 
 (2) In this case $K_X$ is not nef.  
 \end{example}
 
 We pose now the following general questions, (b) being a generalization of Conjecture \ref{co-framed}:
 
  \begin{question}
 Assume that $X$ is a $k$-coframed projective manifold. Does there exist a coframing $V$ (a subbundle 
  $V \cong \hol_X^k$  of $\Omega^1_X$) such that   i) of Proposition \ref{nozero} holds,
  namely the subspace $W =  H^0 (V) \subset H^0(\Omega^1_X) $ defines a quotient Abelian variety $A'$?
  
  (b) If $X$ is projective with $K_X$ nef and  it admits a fibration $f : X \ra A'$ onto an Abelian variety $A'$ with all fibres smooth, is then $f$  a holomorphic bundle (and hence $X$
  is a pseudo-Abelian variety)?

\end{question}

\bigskip

\begin{remark}\label{geometrical}

Question (a) asks, in the case where $X$ is projective, whether  there is a choice of the  subspace $W \subset H^0 (\Omega^1_X)$
 corresponding to an Abelian subvariety of $Alb(X)$, equivalently, whether  
  $\Lam^{2k} (W \oplus \bar{W})$ is a point defined over $\QQ$ in the Grassmann manifold 
  $Grass (2k, H^1(X,\CC)) \subset \PP( \Lam^{2k}(H^1(X,\CC))$.
  
 One may  reduce to the case where  $X$ is defined over an algebraic extension of $\QQ$.
 
\end{remark}

 \section{Mathematical and Historical comments on Baldassarri's paper} 
 
 \subsection{Historical comments}
Baldassarri  \cite{baldassarri} wanted 
  to characterize the smooth projective  manifolds $X$ whose first  
 $h$ canonical systems  $$K_0(X) , \dots, K_{h-1}(X)$$ have degree zero.
 
 At that time, even if Baldassarri in his Ergebnisse book \cite{ergebnisse} (the book was rather influential,
 it was for instance  translated in Russian by Manin)
 was exposing the new methods in the theory of Algebraic Varieties, the concept of canonical systems of all dimension
 was  based on  geometric approaches.
 
 The canonical systems of a manifold (see  \cite{rothgenova}),  defined in a geometric way by  Todd, Eger and later in a simpler way by 
Beniamino Segre \cite{segre}, \cite{segre2} \footnote{ using the embedding covariants for the
 case of the diagonal $\De_X \subset X \times X$: this   approach was also  later 
 followed by Grothendieck.} after proposals made by 
 Severi, were shown in 1955 by  Nakano \cite{nakano} to be the so called  Chern classes of the cotangent bundle
 (see \cite{atiyah} for an historical  account and \cite{fulton} as a general reference):
more precisely, up to sign,  to the system  $K_h(X)$ corresponds the  Chern class
$ c_{n-h}(X)$, where  $n$ is the dimension of  $X$ (and we can consider the Chern classes either as elements
of the Chow ring of $X$, or as integral cohomology classes).

For this  reason   we  rephrase the  questions raised  by  Baldassarri in terms of Chern classes: Baldassarri dealt with the question of characterizing all the
 projective varieties $X$ such that the  rational Chern classes $c_i(X) \in H^*(X, \QQ)$ vanish for $ n-k+1 \leq i \leq n$,
 but not for $i=n-k$.
 
 This question is still wide open, except for the case $k=n$ (cf. Theorem \ref{zerochern}).
 
  As we have shown, already in the surface case  even  the condition that 
 all the Chern numbers are zero  (this means
$K^2_X = e(X)=0$)  does not  imply that $X$ is
a complex torus or a hyperelliptic surface (in the old notation a torus was also called a hyperelliptic surface,
 of grade, or rank, 1). It  implies (see Theorem \ref{dantoni})  that the surface is either a torus or it is 
 birationally (but not necessarily biregularly) covered by a 1-dimensional family of isomorphic elliptic curves.

  Baldassarri's question also suggests (see question (IV) below)
   to classify, if possible,  the  
 varieties  (respectively, the cKM) $X$ whose  integral Chern classes $c_i(X) \in H^*(X, \ZZ)$ vanish for $ n- k + 1 \leq i \leq n$.
 
 Again, this question is still open, even  for the case $k=n$, as we saw in  Part I
 \cite{part1}.

 The paper \cite{baldassarri} by Baldassarri, suggesting that the Pseudo Abelian Varieties of Roth might be the varieties with such vanishing top real Chern classes,  motivated  other   more general questions.
 But, in view of what we have seen already  in the case of surfaces, one has  to  include 
 first of all a condition 
 of minimality of $X$ in a strong sense, for instance  that $K_X$ is nef.
 
  Or require only a birational 
 isomorphism with a product (because a $\PP^1$-bundle over an elliptic curve $T$ is only birational
 to $\PP^1 \times  T$), or that $X$ be only birational  to
 $ (Y \times T)/G$, 
where  the action is only free  in codimension 2  (as in  \cite{haoschreiederbmy},  see especially Cor.1.2 ).

 Todd's review of \cite{baldassarri}  pointed out  a wrong intermediate result in the paper, admitting as counterexample the blow up $X$ of 
$\PP^3$ with centre a smooth curve of genus  $3$, which is a regular manifold $X$ having    $c_3=0$.

Todd's counterexample could somehow be quickly dismissed as  only pointing out the need to assume that $X$ is a minimal manifold, say with $K_X$ nef, as 
we already mentioned.

We observed  however here in this paper that the flaw  is not simply a technical problem, 
 the main claim by Baldassarri that the solutions are Roth's Pseudo-Abelian varieties  
 is incorrect also for minimal manifolds.  Because for instance the class of  Manifolds isogenous to a k-torus product,
 which is a  class of solutions to Baldassarri's problem, is 
 (see section 2) a larger class than the class of
 Roth's Pseudo-Abelian varieties.

 Todd saw however  that the crucial flaw in Baldassarri's argument was the attempt to show (in dimension at least 3) that  manifolds
 with top Chern class equal to zero possess non zero holomorphic 1-forms.
 
What is historically interesting for us is to observe  that the `original sin' of  \cite{baldassarri} was to try
 to extend to higher dimension some wrong results by Enriques,
Dantoni and Roth (indeed the error of Enriques is also reproduced in the classification theorem of Castelnuovo and Enriques
\cite{classificazione}).

 Baldassarri and  Roth were  inspired by a paper by  Dantoni \cite{dantoni}, devoted to the minimal surfaces $X$ with 
$c_2(X)=0$. Dantoni uses surface classification,  especially an article by  Enriques of 1905 \cite{enriques-pg0}, 
claiming that the non ruled surfaces with these properties are
the hyperelliptic manifolds and the `elliptic' surfaces. But `elliptic' for Enriques here does  not have  the same standard meaning 
introduced later by  Kodaira and others: Enriques requires the action of a fixed elliptic curve on $X$,
with all orbits of dimension 1. 
 
 Hence Enriques and  Dantoni in their classification omit to consider  the case of quotients $X= (E \times C)/G$ where the action of the finite group   $G$ is free, 
 of product type, but 
 $G$  does not act on the elliptic curve  $E$  via translations, and moreover  $C$ is a curve of genus  $ g \geq 2$, 
 such that the 
  quotient $C/G$ is an elliptic curve (see section 2, and for instance  \cite{fcbl} or \cite{bea}   for the special case $p_g=0$,
  and  \cite{cb} or \cite{time} for the general case).
  Indeed, in this latter case the automorphism group of 
 $X$ has dimension zero.
 
 Dantoni's paper inspired  Roth \cite{rothPAV} \cite{rothPAV2} who defined the Pseudo-Abelian varieties as the manifolds
 admitting the action of a complex torus of positive dimension   $=k$ having all orbits of dimension  $k$,
 and such that $k$ is maximal with this property.

But for instance, in  \cite{rothhyp}, Roth does not realize about the existence of Hyperelliptic threefolds with 
automorphism groups of dimension zero, and believes that Hyperelliptic threefolds can  only be
 Pseudo-Abelian varieties with 
 $k\geq 1$.
 
 The first  conclusion is simple: the `original sin' was to consider only quotients $X = (T \times Y)/G$ where  $T$ is a torus, 
  $G$ acts freely via a product action, and  $G$ acts on  $T$ via translations. 
 Obviously under these assumptions  $T$ acts on $X$ and the orbits have all the same dimension $ k = dim (T)$!
 
\subsection{Mathematical comments}
Baldassarri' s paper suggested  the  following other questions:

 \begin{question}
 
 (I)  What can be said about a  projective Manifold (respectively, a compact K\"ahler manifold) $X$ with $K_X$ nef and such that  all its  Chern numbers  are  equal to zero? 

 (II)  Is a  projective Manifold  (respectively, a compact K\"ahler manifold) $X$ with $K_X$ nef and such that its  rational Chern classes $c_i(X) \in H^*(X, \RR)$ vanish for $ k + 1 \leq i \leq n$ isogenous to a $k$-Torus Product (respectively, isogenous to a $k$-framed or $k$-coframed manifold)?
 
(III) Same question for a  projective Manifold  (or compact K\"ahler manifold) $X$ with $K_X$ nef and with  all the Chern numbers   equal to zero. 

(IV) What can we say about a projective Manifold  (or cKM) $X$ whose integral Chern classes $c_i(X) \in H^*(X, \ZZ)$ vanish for $ k + 1 \leq i \leq n$?

\end{question}

We shall see in the next section that work of Chad Schoen \cite{schoen} gives a negative answer to Questions (II)
and (III).

 \begin{remark}
a)  We have mentioned  that Baldassarri's  claim is wrong already for surfaces. 
 
b) The claim is also wrong in  dim = 3 since here, as shown in  Prop. 1.5 of Part I, \cite{part1},   there are Hyperelliptic Threefolds whose group of Automorphisms is discrete, hence they are not Pseudo-Abelian. 
 Indeed, if one looks at the paper by Roth \cite{rothhyp} on Hyperelliptic threefolds, one sees that Roth does not consider the
 case of Hyperelliptic Threefolds for which $Aut^0(X) $ consists only of the Identity.
 
 c) In the same paper Roth calls, following Enriques, \cite{enriques-cont},  \cite{enriques-pg0},  `the elliptic surfaces' the surfaces such that $Aut^0(X) $ is an elliptic curve.
 
Here the modern terminology, introduced by Kodaira, differs:
 an elliptic surface  is a surface admitting a fibration $f : S \ra C$ with fibres elliptic curves; in general it  does not possess 
 non-trivial automorphisms.
 \end{remark} 
  
 As we saw in Part I \cite{part1}, however, there are Hyperelliptic Threefolds and Varieties (hence for them  $c_n(X)=0$)
 which  are regular ($H^1(\hol_X)=0$).
 
 Hence the crucial fact that  Baldassarri wants to use, that for $c_n(X)=0$
 we have  an  irregular variety possessing a holomorphic 1-form without zeros has to be taken as an assumption.
 
 To give an idea of the difficulty of the type of questions considered by Baldassarri, let us notice that
 for instance the investigation of varieties such that there is a holomorphic 1-form $\omega \in H^0(\Omega^1_X)$ without zeros
 has been taken up (before our present investigations) in the last years by Schreieder and Hao (\cite{schreieder1} \cite{haoschreieder1}),
 and a classification has been achieved only in dimension 3.
 
 What is more interesting is that, under this much stronger assumption of the existence of such a form $\omega$, the
 results confirms, in a special case, the  result claimed by Baldassarri (see also our Theorem \ref{coframed}).
 
 \begin{theorem}\label{h-s} {\bf (Hao-Schreieder \cite{haoschreieder1})}
 
 Let $X$ be a smooth projective threefold.
 
  Then the   following properties (A), (B), (D) are equivalent.
 
 (A): $X$ admits a holomorphic 1-form $\omega \in H^0(\Omega^1_X)$ without zeros.
 
 (B) : there exists on $X$ a real closed 1-form  without zeros. 
 
 (D):  The minimal model program for $X$ yields a birational morphism $\sigma : X \ra X^{min}$ blowing up 
 smooth elliptic curves which are not contracted by the Albanese map, and 
 such that 
 
 (2) there is a smooth morphism $\pi :  X^{min} \ra A$ to an Abelian variety of positive dimension;
 
 (3) If the Kodaira dimension is non negative, then a finite \'etale cover of $X^{min}$ is a product 
 $ X' \cong A' \times S'$, where $S'$ is smooth projective, and the composite map $A' \cong A'\times \{s'\} \ra A$
 is finite and \'etale.
 
 (4) If the Kodaira dimension is equal to $ - \infty$, then either
 
 (4a) $X^{min}$ has a smooth Del Pezzo fibration over an elliptic curve, or
 
 (4b)  $X^{min}$ has a conic bundle fibration $f :  X^{min} \ra S$ over a smooth surface $S$ satisfying property (A). Moreover,
 either $f$ is smooth, or $A$ is smooth and the degeneracy locus of $f$ is a finite union of elliptic curves which are \'etale over $A$.
 \end{theorem}
 
 \begin{corollary}
 If $X$ is  a smooth projective threefold satisfying  property (A) of  admitting a holomorphic 1-form $\omega \in H^0(\Omega^1_X)$ without zeros, and the Kodaira dimension of $X$ is non negative, then $X^{min}$ is a Pseudo-Abelian variety.

 \end{corollary}
 
 \begin{proof}
 
   We are in case (3), and obviously $A'$ acts on the product $A' \times S'$. 
   
   Since the map $A' \ra A$ is finite and unramified,  it is    a  quotient map,  with Galois group $G$,
 a finite Abelian group of translations of  $A'$.

 We have  $ A' \times S' \ra X^{min} \ra A$, and since $ A' \times S' \ra A'$ is a smooth fibration with fibre $S'$, 
 hence also $X^{min} \ra A$ is a smooth fibration with fibre $S'$.

 Since the pull back of $X^{min} \ra A$ to $A'$ is isomorphic to the product $ A' \times S'$, we conclude 
 that  $A'$ acts on $X^{min} $
 and $ X^{min} = (A' \times S') /G$, where $G$ acts on $A'$ by translations.
 \end{proof}

 \begin{remark}
 Let $X = (A_1 \times A_2 \times C)/G $, where $A_1, A_2$ are elliptic curves,
  $G$ acts on $A_1, A_2, C$, with a  free action  on $A_1$, while  $A_2/G$ and  $C/G$ have genus $0$.
 
 Then $X$ is a MITP of order $k=2$, but it is a Pseudo-Abelian Variety of order $1$,
 such that $h^0(\Omega^1_X) = 1,$ and property (A) is satisfied.
 
 \end{remark}

  The historical conclusion that we can draw is that Baldassarri's paper, even if   vitiated by the `original sin' 
(trying to extend results which were not correct already in small dimension $n=2,3$)
poses some  problems
which, in spite of the tremendous substantial and technical progress which took place in the last 60 or more years,  are still  open and challenging.

The typical  example is  the question of describing the $k$-co-framed manifolds $X$ (see Proposition \ref{nozero}
 and Example \ref{algcoframed}).

 \section{Manifolds with vanishing  Chern numbers}
  
 The title of this section is on purpose ambiguous: one may ask about Manifolds for which  certain Chern numbers
 vanish, or all the way consider Manifolds for which all the Chern numbers are equal to zero.
 
 \medskip
 
 We have given the example of MITP, Manifolds $X$ Isogenous to a k-Torus Product, as a prototype of manifolds
 with all the Chern numbers  equal to zero. 
 
 We have also observed that, at least in dimension 2, 
    all the surfaces with 
 all the Chern numbers $c_1^2= c_2=0$, or the minimal surfaces with $c_2=0$,  in view of  Theorem \ref{dantoni},
 are  MITP' s , if $K_S$ is nef, or birational to  a MITP   if $S$ is ruled.

 Because if  $S$ is minimal and not elliptic ruled,  there exists a Galois \'etale covering $S' \ra S$ such that
 $S' \cong T \times Y$, where $T$ is a complex torus with $ dim(T) >0$.
 
 More generally, we have the class of manifolds $X$  isogenous to
 a partially cotangentially framed manifold $X$.
 
 If we go up to dimension 3, there are three Chern numbers,  $c_1^3,  c_3, $ and $c_1 c_2 = \chi(\hol_X)$.
 
 A recent result by Hao and Schreieder goes in the direction of answering the question, \cite{haoschreiederbmy}
 in the  crucial \footnote{ As explained in loc. cit., the equality implies that $X$ is not of general type and   holds trivially for Kodaira dimension $\leq n-3$; while in  the case of Kodaira dimension $n-2$  the general fibre of the Iitaka fibration is a surface in
 one of the four classes of surfaces of Kodaira dimension $0$, and the situation is clear.} case where the Kodaira dimension is $n-1$ (extending part (1) of Theorem \ref{1-dantoni}):
 
 \begin{theorem}\label{hsbmy} {\bf (Hao-Schreieder)}
 Let $X$ be a minimal model with $dim(X)=n$ and Kodaira dimension $n-1$.
 
 Then $c_1^{n-2} c_2(X)=0$ if and only if $X$ is birational to a quotient $Z = (E \times Y)/G$,
 where
 
 (1) $Z$ has canonical singularities;
 
 (2) $E$ is an elliptic curve and $Y$ is a normal projective variety with $K_Y$ ample;
 
 (3) $G$ acts diagonally, faithfully on each factor, and freely in codimension two on $E \times Y$.
 
 \end{theorem}
 
 Now, the case where $X$ is a threefold of general type (in this case it can be $c_1 c_2=0$, as shown by Ein and Lazarsfeld,
 \cite{el}), is excluded if $X$ is minimal, since then $c_1^3 >0$. 
 
 Hence the missing cases are the cases of Kodaira dimension $0$ and $1$. For Kodaira dimension $0$,
 $c_1(X) =0 \in H^2(X, \QQ)$, hence remains to see what happens for $c_3=0$.
 Some examples of simply connected Calabi-Yau threefolds with $c_3=0$ have been constructed by 
 Chad Schoen \cite{schoen} and other examples were later found by Volker Braun \cite{braun}
 (the latter as hypersurfaces in a toric fourfold).
 
 \subsection{The Schoen Calabi-Yau threefolds $X$ with $c_3(X)=0$}
 We want to discuss now the former examples by Schoen, and show some partial results
 which seem to indicate that they should not 
 be birational to a quotient of a torus product.

 These examples are constructed  as small resolutions of fibre products $ X = S_1 \times_{\PP^1} S_2$
 where $ f_i : S_i \ra \PP^1$ is a rational elliptic surface with a section. Let $ f :X \ra \PP^1$ be the fibre product of $f_1$ and $f_2$.
 
 We are interested in the special case where the critical values of $f_1, f_2$ are different;
 then the fibre product $X$ is smooth, and all the fibres $F$ of $f$ are either a product of two elliptic curves,
 or the product of an elliptic curve with a degenerate fibre. Hence all the fibres $F$ have Euler number 
 $ e(F)=0$, and $ c_3(X) = e (X) =0$.
 
 The canonical divisor of $X$ is trivial, for instance we can take the $f_i = \frac{F_i}{G_i}$ to be given by a pencil of plane cubics
 with simple base points:
 then $X$ is the small resolution of  a hypersurface of bidegree $(3,3)$ $X' \subset \PP^2 \times \PP^2$, 
 $$ X' = \{ (x,y) | F_1 (x) G_2(y) =  F_2 (y) G_1(x) \} $$
 hence $X' $  and $X$ have trivial canonical divisor (see \cite{schoen} page 181). Moreover, essentially 
 by the hyperplane theorem of Lefschetz, $X$ is simply connected (see at any rate (2.1) of \cite{schoen}, page 181).
 
 Thus $X$ is a Calabi-Yau threefold with $c_3(X)=0$.

 \begin{proposition}\label{schoenstrong}
 The Schoen threefolds $X$ (with $K_X$ trivial and  $c_3(X)=0$)
 cannot be birationally  covered by a 1-dimensional family of subvarieties 
 which are isomorphic to a fixed Abelian surface  $T$.

 \end{proposition}
 
 \begin{proof}
Assume that $X$ is birationally covered by a family  $T \times C$ (where by the way $C =\PP^1$
 since $ q(X)=0$). 
 
 For a general $b \in C$,  $T_b := T \times \{b\}$ cannot map to a fibre of $ f :X \ra \PP^1$, since $f$ is the fibre product of $f_1$ and $f_2$
 and we may assume that  the fibrations $f_i$ do not have constant  moduli.
 
 Hence $T_b$, which is a subvariety of $X$,  dominates $\PP^1$ through the morphism $f$. 
 
 But a linear system of dimension one on an Abelian surface
 yields a morphism to $\PP^1$ only if the system is not ample, that is, the fibre is a union of translates of an 
 elliptic curve $E \subset T$. Varying $b$, the elliptic curve $E$ is fixed (since it corresponds to a subgroup of
 the first homology group pf $T$), therefore all the fibres of $f$ contain an elliptic curve isomorphic to $E$.
 This is a contradiction, as the general fibres of $f_1$ and $f_2$ are not even  isogenous to a fixed elliptic curve $E$.
  
 \end{proof}

  \begin{proposition}\label{schoen}
  The Schoen threefolds $X$ do not admit a fibration $\psi : X \ra S$ onto a surface  $S$ such that the general fibre
 is isomorphic to a fixed  elliptic curve $E$.

 \end{proposition}
 \begin{proof}
 Since $X$ is simply connected, and the general fibre of $\psi$  is connected, it follows that $S$ is also simply connected.
 
 Let $D$ be the divisorial part of the set of critical values  of $\psi$, and let $D^*$ be its smooth locus: define 
 $$ S^* : = (S \setminus Sing(D)), \ X^* = \psi^{-1} (S^*).$$
 
 Then we have an exact sequence
 $$ \pi_1(E) \ra \pi_1(X^*) \ra \pi_1(S^*) \ra 1,$$
 and we observe that also $X^*, S^*$ are simply connected, since we have removed a subvariety of
 real codimension 4.
 
 The image of $\pi_1(E) \ra \pi_1(X^*)$ is the quotient of $\pi_1(E)$ by the local monodromies around the 
 irreducible components $D_j$ of $D^*$.
 
 To understand these local monodromies, take a general curve section $C$ of $S$. 
The inverse image of $C$ is an elliptic fibration $\Sigma \ra C$ 
 over $C$ such that all the smooth fibres are isomorphic to $E$.
 
 The fibration is isotrivial, hence $ \Sigma = (C' \times E) / G$, where $C' \ra C$ is Galois with group $G$. 
 
 $G$ acts on $C' \times E$ via a product action. If there are fixpoints for the action of $G$ on $C'$,
 then,  since the quotient is smooth, it follows that the isotropy subgroups act freely on $E$,
 hence by translations. So the local holomorphic monodromies are translations by torsion points,
 and the monodromy acts  trivially  on the first homology of $E$.
 
 The conclusion is that the image  of $\pi_1(E) \ra \pi_1(X^*)$ is an infinite group,
 and this is a contradiction since $ \pi_1(X^*)$ is trivial.

 \end{proof}
 
 The construction of Chad Schoen \cite{schoen} applied to  other elliptic fibrations leads to
 threefolds $X$ with Kodaira dimension 1, and $c_3(X)=0$. We have not yet investigated whether we can achieve with
 this construction trivial Chern number $c_1(X) c_2(X)=0$.
 
 \bigskip

{\bf Acknowledgements:} I would like to thank  Francesco Baldassarri for bringing the paper \cite{baldassarri} by Mario Baldassarri to our attention, thus raising my interest in these questions; he,   Ciro Ciliberto and Flaminio Flamini  provided me with the text of \cite{dantoni}. Thanks to Ciliberto   for a  quick but useful discussion, and  to 
Thomas Peternell for mentioning the article by Braun.

Many thanks to Matthias Sch\"utt for bringing the examples of  \cite{schoen} to my attention and explaining their key features.

Thanks to Pierre Deligne for answering an email query (see Remark \ref{geometrical}).

Thanks to Adriano Tomassini for pointing out Wang's reference \cite{wang}.

 Thanks to the referee for a careful reading of the article and for several helpful comments which helped
to  improve the exposition.

\end{document}